\documentclass[11pt,a4paper]{article}
\usepackage{amsmath}
\usepackage{amsfonts}
\usepackage{amssymb}
\usepackage[numbers]{natbib}
\usepackage{amsthm}
\usepackage{array}
\usepackage{graphicx}
\providecommand{\abs}[1]{\left\lvert#1\right\rvert}

\def \R{\mathbb{R}}
\def \ud{\,\text{d}}

\def \N{\mathbb{N}}

\def \E{\mathbb{E}}

\def \Var{\hbox{{\rm Var}}}

\def \Cov{\hbox{{\rm Cov}}}
\newtheorem{theorem}{Theorem}[section]

\title{Multifidelity variance reduction for pick-freeze Sobol index estimation}
\author{Alexandre Janon\thanks{Laboratoire de Sciences Actuarielle et Financi\`ere, ISFA, Universit\'e Lyon 1, 50 avenue Tony Garnier, 69007 Lyon, France. Homepage: \url{http://isfaserveur.univ-lyon1.fr/\~janona/}}}

\begin{document}
\maketitle
\begin{abstract}
Many mathematical models involve input parameters, which are not precisely known. Global sensitivity analysis aims to identify the parameters whose uncertainty has the largest impact on the variability of a quantity of interest (output of the model). One of the statistical tools used to quantify the influence of each input variable on the output is the Sobol sensitivity index, which can be estimated using a large sample of evaluations of the output. We propose a variance reduction technique, based on the availability of a fast approximation of the output, which can enable significant computational savings when the output is costly to evaluate.
\end{abstract}

\tableofcontents
\section*{Introduction}
Many mathematical models encountered in applied sciences involve a large number of poorly-known parameters as inputs. It is important for the practitioner to assess the impact of this uncertainty on the model output. An aspect of this assessment is sensitivity analysis, which aims to identify the most sensitive parameters, that is, parameters having the largest influence of the output. In global stochastic sensitivity analysis (see for example \cite{saltelli-sensitivity} and references therein) the input variables are assumed to be  independent random variables. Their probability distributions  account for the practitioner's belief about the input uncertainty. This turns the model output into a random variable, whose total variance can be split down into different partial variances (this is the so-called Hoeffding decomposition, see \cite{van2000asymptotic}). Each of these partial variances measures the uncertainty on the output induced by each input variable uncertainty. By considering the ratio of each partial variance to the total variance, we obtain a measure of importance for each input variable that is called the \emph{Sobol index} or \emph{sensitivity index} of the variable \cite{sobol1993}; the most sensitive parameters can then be identified and ranked as the parameters with the largest Sobol indices. 

Once the Sobol indices have been defined, the question of their effective computation or estimation remains open. In practice, one has to estimate (in a statistical sense) those indices using a finite sample (of size typically in the order of hundreds of thousands) of evaluations of model outputs \cite{helton2006survey}. Indeed, many Monte Carlo or quasi Monte Carlo approaches have been developed by the experimental sciences and engineering communities. Such an approach is the Sobol pick-freeze (SPF) scheme  (see \cite{sobol1993,sobol2001global}). In SPF a Sobol index is viewed as the regression coefficient between the output of the model and its pick-freezed replication. This replication is obtained by holding the value of the variable of interest (frozen variable) and by sampling the other variables (picked variables). The sampled replications are then combined to produce an estimator of the Sobol index. 

The SPF method requires many (typically, around one thousand times the number of input variables) evaluations of the model output. In many interesting cases, an evaluation of the model output is made by a complex computer code (for instance, a numerical partial differential equation solving algorithm) whose running time is not negligible (typically in the order of a second or a minute) for one single evaluation. When thousands of such evaluations have to be made, one generally replaces the original {\it exact} model by a faster-to-run \emph{metamodel} (also known in the literature as \emph{surrogate model} or \emph{response surface} \cite{box1987empirical}) which is an approximation of the true model. Well-known metamodels include Kriging \cite{sant:will:notz:2003}, polynomial chaos expansion \cite{sudret2008global} and reduced bases \cite{nguyen2005certified,janon2011certified}, to name a few. From a multifidelity point of view, the metamodel can also be viewed as a ``coarse'' (low-fidelity) version of the code; the metamodel, seen as a coarse version, may also be a ``degraded'' version of the code: for instance, it may be a solver for a simplified model (either mathematically simplified, or discretized on a coarser grid), an integrator for a function of lesser precision, or an optimizer stopped before its full convergence. In this paper, we designate by ``coarse approximation'' any of the above approximations (metamodels and degraded versions).   When using a coarse approximation for sensitivity analysis, the original model is generally used only to define the metamodel, and not to perform the Sobol index estimation. This leads to a necessity of measuring the difference between the model and its approximation in order to certify the sensitivity index estimation \cite{janon2011certified,jaal}. To our best knowledge, no approach for using both metamodel and model evaluations to estimate Sobol indices have been proposed yet. 

In this work, we propose an approach, based on the asymptotic properties of the SPF scheme studied in \cite{jaal} to optimally combine evaluations of the original model and evaluations of its approximation, in order to produce an asymptotically-justified confidence interval for the Sobol index of the original model. Our approach is inspired by the quasi-control variate method \cite{emsermann2002improving} which has been developed for Monte-Carlo estimation of means. 

This paper is organized as follows: in the first section, we begin by setting up the notations and the context of the paper. Then we define the Sobol index estimator we wish to study. The main result is Theorem \ref{t:theo}, which provides an asymptotic method to estimate a confidence interval for a Sobol index. The second section is a numerical illustration on a particular (but representative) kind of model output. 

\section{Motivation and definition of the estimator}
\subsection{Notation and context}
We begin by setting up the usual notations in the sensitivity analysis contexts. The output of interest is a random variable $Y$, which is a deterministic function $\eta: \R^p \rightarrow \R$ of the random inputs $X \in \R^{p_1}$ and $Z \in \R^{p_2}$:
\[ Y = \eta(X, Z), \]
where $p_1$ and $p_2$ are integers, and $p=p_1+p_2$.

We assume that $X$ and $Z$ are independent random variables and that $Y$ has a finite and nonzero variance. We are interested in the (closed) Sobol index \cite{saltelli2004practice} with respect to $X$, defined by:
\[ S = \frac{\Var(\E(Y|X))}{\Var Y}. \]

This index, which is between 0 and 1, quantifies the influence of the $X$ input on the output $Y$: a value of $S$ that is close to $1$ indicates that $X$ is highly influential on $Y$. 

The pick-freeze method \cite{jaal} expresses $S$ using a covariance:
\[ S = \frac{\Cov(Y,Y')}{\Var Y} \;\; \textrm{ for } \;\; Y'=f(X,Z'), \]
where $Z'$ is an independent copy of $Z$.

This expression leads to different Monte-Carlo estimators of $S$. For instance, the following estimator is studied in \cite{jaal}:
\[ T_N^\eta = \frac{ \frac1N \sum Y_i Y_i' - \left( \frac1N \sum \frac{Y_i+Y_i'}{2} \right)^2 }{ \frac{1}{N}\sum  \Big[\frac{Y_i^2+(Y_i')^2}{2}\Big] - \left( \frac{1}{N} \sum  \Big[\frac{Y_i+Y_i'}{2}\Big] \right)^2 }, \]
where, $(Y_i)_{i=1,\ldots,N}$ and $(Y_i')_{i=1,\ldots,N}$ are independent samples of $Y$ (resp. $Y'$), and, as in the rest of the paper, all sums are for $i$ from 1 to $N$.

It is shown [op.cit., Proposition 2.2] that $(T_N)_N$ is asymptotically normal, with variance $\sigma_{T,\eta}^2/N$, where:
\begin{equation}
\label{e:defsigmaTeta}
\sigma_{T,\eta}^2 = \frac{\Var\left( (Y-\E(Y))(Y'-\E(Y)) - S/2 \left( (Y-\E(Y))^2+(Y'-\E(Y))^2 \right) \right)}{(\Var(Y))^2}, 
\end{equation}
and [op.cit., Proposition 2.5] that this asymptotic variance is minimal among regular estimators that are functions of realizations of exchangeable $(Y,Y')$ pairs.

Note that a realization of the $T_N^\eta$ estimator, for a finite sample size $N$, can be computed by making $2N$ evaluations of the $\eta$ function. 

In this paper, we suppose that we can evaluate, in addition to the $\eta$ function, an \emph{approximation} $\eta_c: \R^p \rightarrow \R$ of the $\eta$ function (the $c$ index is for \emph{coarse}). 

The usage of such an approximation has been motivated in the Introduction. A concrete and ubiquitous example of $\eta$ and $\eta_c$ will be presented in the next section. In the following section, we motivate and study our variance-reduced estimator of $S$.  

\subsection{Variance-reduced estimator}
Let:
\[ Y_c = \eta_c(X,Z), \;\; Y_c' = \eta_c(X,Z'), \]
and $(Y_{ci})_{i=1,\ldots,N}$, $(Y_{ci}')_{i=1,\ldots,N}$ be $N$-samples of $Y_c$ (resp. $Y_c'$).

The estimator:
\[ T_N = \frac{ \frac1N \sum Y_{ci} Y_{ci}' - \left( \frac1N \sum \frac{Y_{ci}+Y_{ci}'}{2} \right)^2 }{ \frac{1}{N}\sum  \Big[\frac{Y_{ci}^2+(Y_{ci}')^2}{2}\Big] - \left( \frac{1}{N} \sum  \Big[\frac{Y_{ci}+Y_{ci}'}{2}\Big] \right)^2 } \]
consistently estimates the Sobol index of the coarse model:
\[ S_c = \frac{\Var(\E(Y_c|X))}{\Var Y_c}. \]
by using $2N$ evaluations of $\eta_c$.

As mentioned in the introduction, our objective is to combine evaluations of $\eta$ and $\eta_c$ to estimate $S$ at a smaller cost than an estimation that would be performed from evaluations of $\eta$ only.

We take a function $\psi: \N \rightarrow \N$. 

It is clear that the estimator $E_N$ defined by:
\[ E_N = T_{\psi(N)}^\eta - T_{\psi(N)}  \]
consistently estimates $E = S - S_c$, and that a realization of $E_N$ can be obtained using $2 \psi(N)$ evaluations of $\eta_c$ and $2 \psi(N)$ evaluations of $\eta$.

We propose a natural estimator of $S$ based on $T_N$ and $E_N$, inspired by the quasi-control variate method \cite{emsermann2002improving}, is thus:
\[ V_N = T_N + E_N. \]
This estimator can be computed by making $2(N+\psi(N))$ evaluations of $\eta_c$ and $2\psi(N)$ evaluations of $\eta$. As an evaluation of $\eta$ is more costly than one of $\eta_c$, one can expect a computational gain if $\psi(N)\leq N$, and if the asymptotic variance of $(V_N)$ 
is less than the asymptotic variance of $(T_{\psi(N)}^\eta)$, so that asymptotic confidence intervals built upon $V_N$ are more precise than those built on $T_{\psi(N)}^\eta$ alone.

The following theorem gives a method for estimating (conservative) asymptotic confidence intervals using $V_N$. We denote by $\Phi$ the cumulative distribution function of the Gaussian with zero mean and unit variance, and by $\Phi^{-1}$ its inverse.

\begin{theorem}
\label{t:theo}
Suppose that $\lim_{N\rightarrow +\infty} \psi(N)=+\infty$. 

Then, for any $\alpha_e$ and $\alpha_c$ in $]0,1[$:
\[ \lim_{N\rightarrow +\infty} P\left(  \abs{ S - V_N } \leq q(\alpha_e) \frac{\sigma_{e}}{\sqrt{\psi(N)}} + q(\alpha_c) \frac{\sigma_c}{\sqrt N} \right) \geq 1-(\alpha_e+\alpha_c), \]
for:
\[ q(a) =  \Phi^{-1}(1-a/2),  \;\; 
\sigma_c^2 =  \frac{ \Var\left( A_c - B_c/2 \right) }{(\Var Y_c)^2 },
\]
\begin{align*}
\sigma_e^2 &= \sigma_c^2 + \frac{ \Var\left( A - B/2 \right) }{(\Var Y)^2 } 
- \frac{ 2 \Cov(A, A_c) - \left( \Cov(A, B_c) + \Cov(B, A_c) \right) + \Cov(B, B_c)/2  }{\Var Y \, \Var Y_c}, 
\end{align*}
where $A$, $B$, $A_c$, $B_c$ are the following random variables:
\[ A=(Y-\E(Y))(Y'-\E(Y)),\;\; B=S \left[ (Y-\E(Y))^2 + (Y'-\E(Y))^2 \right], \]
\[ A_c=(Y_c-\E(Y_c))(Y_c'-\E(Y_c)),\;\; B_c=S_c \left[ (Y_c-\E(Y_c))^2 + (Y_c'-\E(Y_c))^2 \right]. \]

The same holds when $\sigma_c$ and $\sigma_e$ are replaced by any consistent estimators.
\end{theorem}

\begin{proof}[Sketch of proof]
Follow the proof of \cite{jaal}, Proposition 2.2, and apply the $\delta$-method to $(T_N, T_{\psi(N)}^\eta)$ to get the asymptotic variance of $(E_N)$.

Then use that for any $\epsilon_1, \epsilon_2$,
\[ \left\{ \abs{V_N-S} \geq \epsilon_1+\epsilon_2 \right\} \subseteq \left\{ \abs{T_N-S_c}\geq \epsilon_1 \right\} \cup
\left\{ \abs{E_N-E}\geq \epsilon_2 \right\}. \qedhere \]
\end{proof}

\subsection{Choice of $\psi$, $\alpha_e$ and $\alpha_c$}
\label{ss:choice}
To convert the theorem above into a practical procedure, it remains to choose the parameters $\psi$, $\alpha_e$ and $\alpha_c$, so as to minimize the overall computational time.

We will assume that one evaluation of $\eta$ as a unit cost, and that an evaluation of $\eta_c$ has cost $0<\rho<1$. We also set $\psi(N)=\lceil \mu N \rceil$, where  $\mu\in]0,1[$ is to be found, and $\lceil \cdot \rceil$ is the ``ceiling'' function.

We choose a target risk level $\alpha \in ]0,1[$ and a target length $L$ for the confidence interval of Theorem \ref{t:theo}.

It is clear these constraints force $\alpha_c$ in function of $\alpha_e$ and $\alpha$:
\[ \alpha_c = \alpha^*(\alpha_e) = 1 - (\alpha+\alpha_e), \]
and that $N$  has to satisfy:
\[ N \geq N^*(\alpha_e, \mu) := \frac{4}{L^2} \left( \frac{q(\alpha_e) \sigma_e}{\sqrt \mu} + q(\alpha^*(\alpha_e)) \sigma_c \right)^2 \]
We approximate $\psi(N^*)$ by $\mu N^*$. The cost of the required evaluations of $\eta$ and $\eta_c$ is thus, in the general case:
\[ \text{Cost}(\alpha_e,\mu) = 2 N^*(\alpha_e, \mu) \left( 2 \mu + \rho \right), \]
corresponding to the $\psi(N)$ evaluations of $\eta$ and the $\psi(N)+N$ evaluations of $\eta_c$.

However, in some settings, the computations made to compute $\eta_c$ can be reused to compute $\eta$, allowing to evaluate $\eta_c$ \underline{and} $\eta$ for a unit cost, leading to:
\[ \textrm{Cost}_{\textrm{Hier}}(\alpha_e,\mu) = 2 N^*(\alpha_e,\mu) \left( \mu + \rho \right). \]
Such a ``hierarchical'' property is beneficial to our estimation scheme and occurs naturally for some $\eta$, as we will see in the numerical illustration section.

Now, one would obviously choose $\alpha_e$ and $\mu$ so as to minimize the cost $\text{Cost}(\alpha_e,\mu)$ (or, depending on the case at hand, $\textrm{Cost}_{\text{Hier}}(\alpha_e,\mu)$). In practice, this is not possible, as $\sigma_e$ and $\sigma_c$ are unknown. Hence, approximately optimal parameters are found by empirically estimating these quantities, based on a small sample of realizations of $Y$, $Y'$, $Y_c$ and $Y_c'$. This gives rise to $\widehat\alpha_e^*$ and $\widehat\mu^*$, and an estimated optimal costs: 
\[ \widehat{\textrm{Cost}}(\widehat \alpha_e^*, \widehat\mu^*) \text{ and } \widehat{\textrm{Cost}_{\textrm{Hier}}}(\widehat \alpha_e^*, \widehat\mu^*) . \]

\section{Numerical illustration}
\subsection{Model set-up}
In financial mathematics, the Heston model \cite{heston1993closed} is the following stochastic differential model for the price of a risky asset $(S_t)_{t\ge 0}$ as function of the time $t>0$:
\[
\left\{
	\begin{array}{l}
	\ud S_t = \mu S_t \ud t + \sqrt{\nu_t} S_t \ud W_t^1 \\
	\ud \nu_t = \kappa (\theta - \nu_t) \ud t + \xi \sqrt{\nu_t} \ud W_t^2
	\end{array}
\right.,
\]
where $(W_t^1)_{t\ge 0}$ and $(W_t^2)_{t\ge 0}$ are standard Brownian motions (under the risk-neutral probability measure $Q$) whose correlation is $r \in [0,1]$.

We are interested in the price of an European call option of maturity $T$ and strike $K$, which is given by $e^{-R T}\E_Q\left( \left( S_T-K \right)_+ \right)$.

Although a semi-analytical formula is available for the fast computation of this expectation (such a formula may not exist for more complex dynamics of the underlying asset, or for exotic options), we will use a numerical approximation so as to illustrate our methodology on a realistic model example. The expectation is approached by the following Monte-Carlo procedure:
\[ \eta(\nu_0, \kappa, \theta, r, \xi, R, S_0, T, K) = \frac{ e^{-R T} }{ M } \sum_{j=1}^M ( S_{T,j} - K )_+ \]
with an Euler-Maruyama approximation of $(S_t, \nu_t)_{t\in[0,T]}$ with timestep $h>0$: for $j=1,\ldots,M$ and $t=1, ..., T/h$:
\[
\left\{
	\begin{array}{l}
	S_{0, j}=S_0\\
	\nu_{0, j}=\nu_0 \\
	S_{t h, j} = S_{(t-1)h, j} \left( 1 + R h + \sqrt{\nu_{(t-1)h, j}} \sqrt{h} \Delta W_{t, j}^1 \right) \\
	\nu_{t h, j} = \nu_{(t-1)h, j} \left( 1+\kappa(\theta-\nu_{(t-1)h, j}) h + \xi \sqrt{\nu_{(t-1)h, j}} \sqrt{h} (r \Delta W_{t,j}^1 + \sqrt{1-r^2} \Delta W_{t,j}^2) \right)
	\end{array}
	\right.
\]
where $h>0$ is the time discretization parameter, and $(\Delta W_{t,j}^1, \Delta W_{t,j}^2)$ are indepedent realizations of a standard Gaussian random variable.

We fix $S_0=60$ (the initial price of the asset), $T=0.25$, $K=30$, as well as the discretization parameters $h=.001$ and $M=10000$. The uncertain parameters are $X=(\nu_0)$ and $Z=(\kappa, \theta, r, \xi, R)$, which are given the uniform distribution probabilities summarized in Table \ref{t:1}.

\begin{table}
\centering
\begin{tabular}{|l|l|l|l|}
\hline Name &    Interpretation           & Min. & Max. \\
\hline $\nu_0$ & Initial volatility       & .2     & .25     \\
\hline $\kappa$& Volatility convergence rate& 0   &  3    \\
\hline $\theta$& Volatility limit         &  .2    &  .22    \\
\hline $r$     & Correlation between Brownians &  -1  &  1   \\
\hline $\xi$   & Volatility of the volatility &  0   &  .4    \\
\hline $R$       & Risk-free rate           &  .08    &  1.1    \\
\hline
\end{tabular}
\caption{Distributions and interpretations of the input parameters. }
\label{t:1}
\end{table}

The coarse approximation uses a reduced number $m$ of simulated trajectories to compute the empirical mean:
\[ \eta_c(\nu_0, \kappa, \theta, R, \xi, R, S_0, T, K) = \frac{ e^{-R T} }{ m } \sum_{j=1}^m ( S_{T,j} - K )_+ \]
Note that for computing $\eta_c$, the same time discretization parameter $h$, as well as the same simulated Brownian increments $\Delta W^{1,2}$ are kept, hence our approximation is ``hierarchical'' in the sense of Subsection \ref{ss:choice}.

We chose $m=5000$, so that $\rho=m/M=1/2$.

\subsection{Results and discussion}
We estimated $\sigma_c$ and $\sigma_e$ based on a sample of $n=100$ realizations of each variable $Y, Y', Y_c$ and $Y_c'$. The estimates are:
\[ \widehat{\sigma_c}=.9017 \;\; \widehat{\sigma_e}=.4909.  \]
For comparison purposes, we also estimated $\sigma_{T,\eta}$:
\[ \widehat{\sigma_{T,\eta}}=.8491. \]

We are interested in the (estimated) \emph{relative efficiency} of the confidence intervals based on our variance-reduced estimator, as compared with those based on $T^\eta$, that is:
\[ \widehat{\text{Eff}} = 1 - \frac{ \widehat {\textrm{Cost}_{\textrm{Hier}}}(\widehat \alpha_e^*, \widehat\mu^*) }{ \widehat{\textrm{ClassicalCost} }},  \]
where $\widehat {\textrm{ClassicalCost}}$ is the cost of the $\eta$ evaluations necessary to produce an asymptotic confidence interval of fixed length $L$ using only the $T^\eta$ estimator:
\[ \widehat {\textrm{ClassicalCost}} = 2 \frac{4}{L^2} (q(\alpha)\widehat\sigma_{T,\eta})^2. \]
As the denominator and the numerator of $\text{Eff}$ are proportional to $L^2$, the relative efficiency is independent of the target length of the confidence interval $L$. 

In Figure \ref{f:1}, we plot the estimated relative efficiency of our variance-reduced estimator, as function of the target risk level $\alpha$. 

\begin{figure}
\centering
\includegraphics[width=10cm]{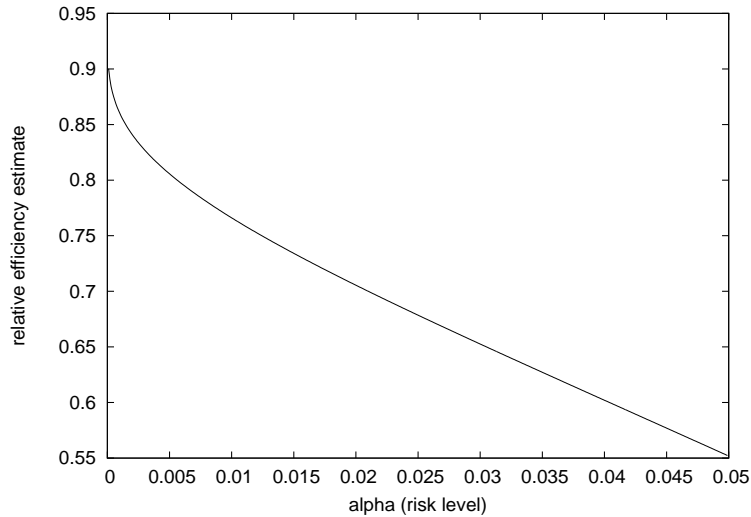}
\caption{Estimates of relative efficiencies, for various values of $\alpha \in [0.0001, 0.05] $. }
\label{f:1}
\end{figure}

We see that, based on empirical estimations, our variance reduction enables an interesting reduction of the computational cost by more than $50\%$ for $\alpha=0.05$, and this reduction is even more significative for small risk levels (up to $90\%$ for $\alpha=0.0001$). 

\bibliographystyle{plain}
\bibliography{biblio,bibi}

\bigskip

\textbf{Acknowledgements. } This work has been partially supported by the French National Research Agency (ANR) through COSINUS program (project COSTA-BRAVA nr. ANR-09-COSI-015).

\end{document}